\documentclass[preprint,12pt]{elsarticle}

\usepackage{amsthm,amsmath,amsfonts,amssymb}
\usepackage{mathtools}
\usepackage{graphicx}
\usepackage[colorlinks=true,linkcolor=blue,citecolor=blue,urlcolor=blue]{hyperref}

\journal{Statistics \& Probability Letters}
\biboptions{sort&compress}
\hypersetup{
  pdftitle={Crossing the Kolmogorov-Smirnov Boundary: Exact Tails, Sharp Bounds, and Broken Pivots},
  pdfauthor={Elvis Han Cui, Yihao Li, Zhuang Liu}
}

\newcommand{\Cov}{\operatorname{Cov}}

\theoremstyle{plain}

\newtheorem{theorem}{Theorem}[section]

\newtheorem{corollary}[theorem]{Corollary}
\newtheorem{proposition}[theorem]{Proposition}

\theoremstyle{remark}

\newtheorem{remark}[theorem]{Remark}
\newtheorem*{example}{Example}

\begin{document}

\begin{frontmatter}
\title{Crossing the Kolmogorov--Smirnov Boundary: Exact Tails, Sharp Bounds, and Broken Pivots}

\author[zju,ucla]{Elvis Han Cui\corref{cor1}}
\ead{elviscuihan@g.ucla.edu}
\cortext[cor1]{Corresponding author.}

\author[zju,genentech]{Yihao Li}
\ead{li.yihao@gene.com}

\author[zju,google]{Zhuang Liu}
\ead{zhuangl@google.com}

\affiliation[zju]{organization={College of Environmental and Resource Sciences, Zhejiang University},
            city={Hangzhou},
            country={China}}

\affiliation[ucla]{organization={Department of Biostatistics, University of California, Los Angeles},
            state={CA},
            country={USA}}

\affiliation[genentech]{organization={Genentech},
            city={South San Francisco},
            state={CA},
            country={USA}}

\affiliation[google]{organization={Google},
            city={San Francisco Bay Area},
            state={CA},
            country={USA}}

\begin{abstract}
The Kolmogorov-Smirnov statistic is usually introduced as a supremum, but its finite-sample behavior is governed by a more local question: where does the empirical process first cross a boundary? This letter gives a partial answer through a finite-sample crossing ledger. The ledger rewrites the Smirnov-Birnbaum-Tingey one-sample formula as an explicit hitting-time law and yields a stable log-scale tail evaluator. For two samples, it gives one-wall and two-wall exact lattice recursions for arbitrary sample sizes, with the balanced reflection formula appearing as a special closed form. The same viewpoint explains the Dvoretzky-Kiefer-Wolfowitz-Massart inequality as an exponential compression of exact crossing sums and shows where exact distribution-free counting stops: under a composite null, fitted parameters change the path itself. Simulations and two small data diagnostics illustrate the resulting calibration warning.
\end{abstract}


\end{frontmatter}
\section{Introduction}\label{sec:introduction}

The Kolmogorov-Smirnov (KS) statistic has a deceptively simple face: take the largest vertical gap between an empirical distribution function and a reference distribution, or between two empirical distribution functions. That single supremum has made the KS test a durable nonparametric tool for goodness-of-fit and two-sample comparison \cite{berger2014kolmogorov}. Yet the supremum is the end of the story, not the beginning. In a finite sample, the statistic is created by a concrete event along a path:
\[
\textit{where does the empirical process first cross a boundary?}
\]
This is the question we use as the organizing principle of the paper. Our
answer is deliberately partial: we identify the parts of the KS family where
crossing can be counted exactly, the parts where a sharp inequality is the
right compression, and the point at which fitted parameters end the
distribution-free calculation.

The crossing question is worth asking because the KS statistic is not merely a limit theorem in disguise. Figure~\ref{fig:zn-crossing} shows the basic geometry. In one sample, after the probability integral transform, the empirical process drifts deterministically between observations and jumps at data points. A boundary crossing therefore happens only through a finite set of possible hitting locations. This turns the one-sided KS tail probability into a finite hitting-time calculation. In two samples, the same idea becomes a ballot-path problem: pooled orderings of the two samples can be read as paths, and a boundary crossing can be counted or recursively excluded.

Classical KS results are often presented in separate languages: order-statistic integrals for one-sample exact laws, lattice paths for two-sample laws, empirical-process bounds for DKWM, and Gaussian projections for fitted models. Our aim is not to replace those theories with one grand formula. It is to give a compact operational answer to the question above: once a KS statistic is viewed as a crossing event, which parts can still be computed exactly, which parts should be compressed into sharp inequalities, and where does distribution-free counting genuinely stop?

The original contribution is this finite-sample crossing ledger. First, it states the Smirnov-Birnbaum-Tingey formula as a first-hitting distribution and gives a log-scale evaluator for the exact one-sided tail. Second, it gives one-wall and two-wall dynamic-programming recursions for arbitrary two-sample sizes $n$ and $m$, so the balanced reflection formula becomes a special case rather than the starting point. Third, it uses the same ledger to separate three regimes: exact finite enumeration, DKWM-type exponential compression, and composite-null testing, where estimated parameters break the pivotal path and force refitted simulation or empirical-process approximation.

\section{One Sample: The First Crossing Time}\label{sec:one-sample-one-sided}

Let $X_1,\ldots,X_n$ be a random sample with continuous distribution function $F$ and let $F_n$ be its empirical counterpart:
\[
F_n(t)=\frac{1}{n}\sum_{i=1}^n\mathbb{I}(X_i\le t).
\]
Define the normalized empirical process by
\[
Z_n(t)=\sqrt{n}\left(F_n-F\right)(t)
\]
and let $D_n^{+}$ and $D_n^-$ denote the one-sided Kolmogorov-Smirnov statistics:
\[
D_n^+=\sup_{t}\frac{Z_n(t)}{\sqrt{n}}=\sup_t(F_n(t)-F(t)),
\]
\[
D_n^-=\sup_{t}\frac{-Z_n(t)}{\sqrt{n}}=\sup_t(F(t)-F_n(t)).
\]

Here the supremum becomes literal. Between observations the centered empirical
process drifts downward; at an observation it jumps upward. A lower-boundary
crossing therefore cannot occur everywhere on the line, but only at a small
set of lattice points. The next theorem is the classical exact answer, written
in this first-passage language.

\begin{figure}[t]
\centering
\includegraphics[width=0.82\textwidth]{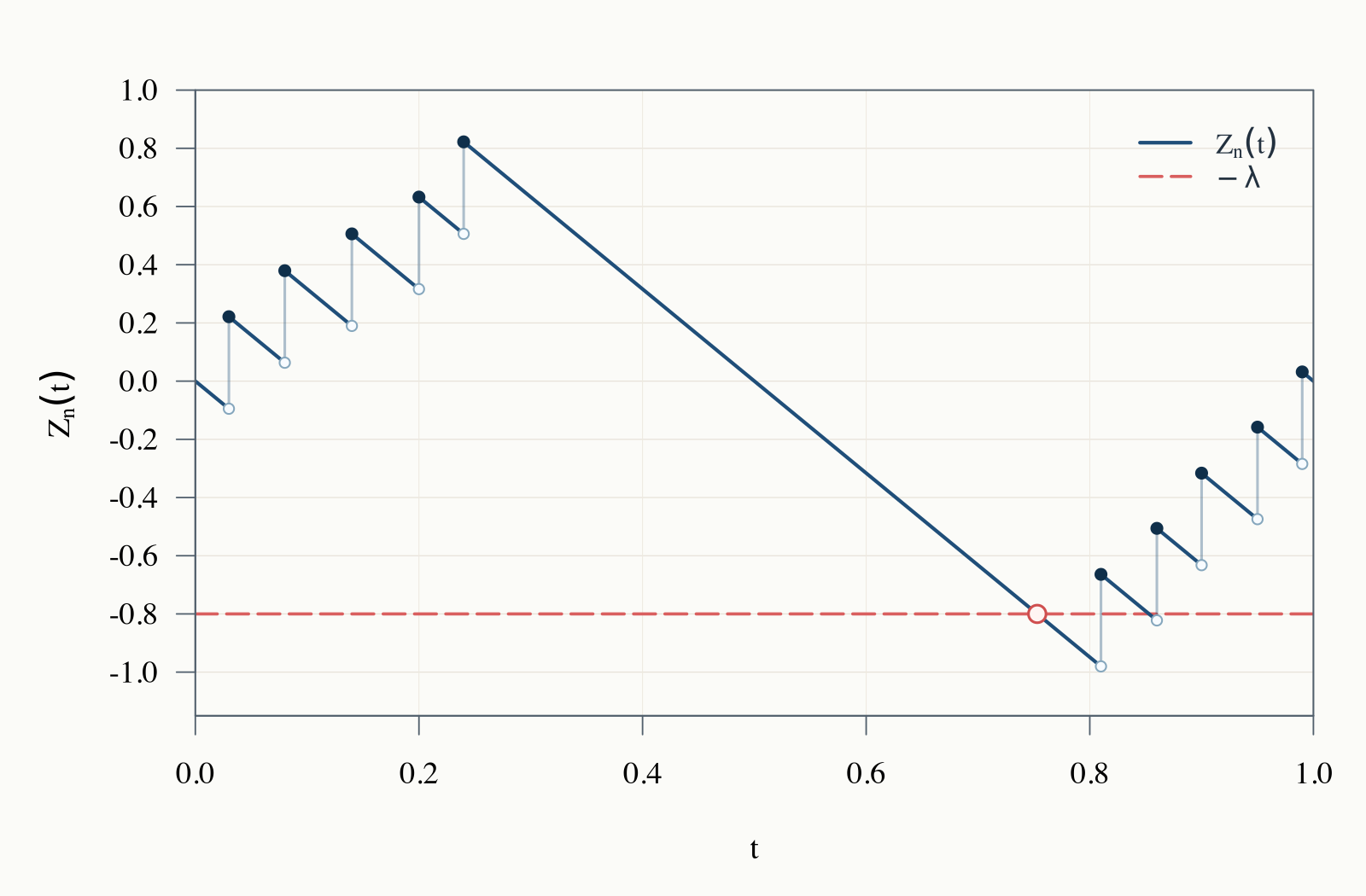}
\caption{A one-sample empirical process on the uniform scale for $n=10$. The
blue path is $Z_n(t)=\sqrt n\{F_n(t)-t\}$; it drifts downward between
observations and jumps upward at sample points. The red line is the lower
boundary $-\lambda$ with $\lambda=0.8$. The highlighted point is the first
boundary crossing, reducing the KS tail calculation to the finite hitting-time
question $t=\lambda/\sqrt n+j/n$.}
\label{fig:zn-crossing}
\end{figure}

\begin{theorem}[Smirnov-Birnbaum-Tingey \cite{birnbaum1951one}]
Let $0<\epsilon<1$ and set $\lambda=\sqrt{n}\epsilon$. Define
\[
\tau_n(\lambda)=\inf_t\{t:Z_n(t)\le-\lambda\}
\]
as the first time that $Z_n(t)$ hits $-\lambda$. Then
\begin{align}
    \mathbb{P}\left(D_n^->\epsilon\right)
    &=(1-\epsilon)^n+\epsilon\sum_{j=1}^{\lfloor n(1-\epsilon)\rfloor}\binom{n}{j}
    \left(\epsilon+\frac{j}{n}\right)^{j-1}
    \left(1-\epsilon-\frac{j}{n}\right)^{n-j}.
\end{align}
On the uniform scale, the corresponding first-passage probabilities are
\begin{align}
    \mathbb{P}\{\tau_n(\lambda)=\epsilon\}
    &=(1-\epsilon)^n,\\
    \mathbb{P}\left\{\tau_n(\lambda)=\epsilon+\frac{j}{n}\right\}
    &=\epsilon\binom{n}{j}\left(\epsilon+\frac{j}{n}\right)^{j-1}
    \left(1-\epsilon-\frac{j}{n}\right)^{n-j}
\end{align}
for $j=1,\ldots,\lfloor n(1-\epsilon)\rfloor$.
\end{theorem}

\begin{remark}
The reduction to $F(t)=t$ is distribution-free: after the probability integral transform, the distributions of $D_n^+$ and $D_n^-$ do not depend on $F$. On the original scale, the hitting locations are obtained by applying the generalized inverse $F^{-1}$ to the uniform-scale values $\epsilon+j/n$.
\end{remark}

\begin{remark}
The hitting time $\tau_n(\lambda)$ is discrete, although $\tau_n(\lambda)=+\infty$ may occur with positive probability. The process $Z_n(t)$ decreases between jump points; at each observation $X_i$, it jumps by $1/\sqrt{n}$. If $Z_n(t)$ does not cross below $-\lambda$ at $X_{(i)}$, the $i$-th order statistic, it must wait another $1/n$ in transformed time before crossing becomes possible. The first possible hitting time is $\epsilon=\lambda/\sqrt{n}$, corresponding to no observations below $\epsilon$. Thus the equation
\[
F_n(t)-t=-\epsilon
\]
has only finitely many possible solutions, all of the form $\epsilon+j/n$.
\end{remark}
\begin{proof}
It is enough to work on the uniform scale. Let
$U_{(1)}<\cdots<U_{(n)}$ be the order statistics of independent
$\operatorname{Unif}(0,1)$ variables, with $U_{(0)}=0$ and $U_{(n+1)}=1$.
On $[U_{(j)},U_{(j+1)})$, the empirical distribution function is $j/n$.
Thus the first crossing occurs at $t_j=\epsilon+j/n$ exactly when
\[
U_{(i)}\le \epsilon+\frac{i-1}{n}\quad(i=1,\ldots,j),
\qquad
U_{(j+1)}> \epsilon+\frac{j}{n}.
\]
The case $j=0$ gives $(1-\epsilon)^n$. For $j\ge1$, the probability is
$n!$ times the ordered-simplex volume cut out by these inequalities. The
Birnbaum-Tingey simplex identity \cite{birnbaum1951one} gives this volume as
\[
\frac{\epsilon}{j!(n-j)!}
\left(\epsilon+\frac{j}{n}\right)^{j-1}
\left(1-\epsilon-\frac{j}{n}\right)^{n-j}.
\]
Multiplying by $n!$ gives the displayed mass at $\epsilon+j/n$, and summing
over $j=0,\ldots,\lfloor n(1-\epsilon)\rfloor$ gives the tail probability.
The ordered-simplex calculation is written out in the supplementary appendix.
\end{proof}

\begin{remark}[Log-scale exact calculator]
The preceding formula is a practical $O(n)$ evaluator, not only a closed
form. Let $J=\lfloor n(1-\epsilon)\rfloor$ and define
\[
\ell_0=n\log(1-\epsilon),
\]
and, for $j=1,\ldots,J$,
\[
\ell_j=\log\epsilon+\log\binom{n}{j}
 +(j-1)\log\left(\epsilon+\frac{j}{n}\right)
 +(n-j)\log\left(1-\epsilon-\frac{j}{n}\right),
\]
with $\ell_j=-\infty$ whenever the last logarithm has zero argument. Then
\[
\log \mathbb{P}(D_n^->\epsilon)=\log\sum_{j=0}^J e^{\ell_j}.
\]
This log-sum-exp form avoids overflow and cancellation in direct evaluation;
an implementation is included in the source package.
\end{remark}

\section{Two Samples: Recursion Before Reflection}
The two-sample statistic has the same crossing flavor, but the path now comes
from the pooled ordering of two samples. With unequal sample sizes the boundary
is weighted, so a direct reflection formula is no longer the cleanest object.
The exact finite-sample calculation is instead a lattice recursion; reflection
appears when the lattice is balanced. The next two propositions are the
algorithmic core of the ledger: they count the paths that avoid one wall, or
both walls, for arbitrary $n$ and $m$.

Let $X_1,\ldots,X_n$ be a random sample with continuous distribution function $F$ and let $Y_1,\ldots,Y_m$ be another random sample with continuous distribution function $G$. Let $F_n$ and $G_m$ be the empirical counterparts of $F$ and $G$:
\[
F_n(t)=\frac{1}{n}\sum_{i=1}^n\mathbb{I}(X_i\le t),
\qquad
G_m(t)=\frac{1}{m}\sum_{j=1}^m\mathbb{I}(Y_j\le t).
\]

\begin{proposition}[Exact two-sample crossing recursion \cite{feller1968introduction}]
Let
\[
D_{n,m}^{+}=\sup_t\{F_n(t)-G_m(t)\}.
\]
Under $F=G$, fix $\delta>0$ and define $A_{0,0}=1$. For $0\le i\le n$ and
$0\le j\le m$, set $A_{i,j}=0$ outside the rectangle and, for $(i,j)\ne(0,0)$,
\[
A_{i,j}=
\mathbf{1}\left\{\frac{i}{n}-\frac{j}{m}<\delta\right\}
\left(A_{i-1,j}+A_{i,j-1}\right).
\]
Then
\[
\mathbb{P}\left(D_{n,m}^{+}<\delta\right)
=\frac{A_{n,m}}{\binom{n+m}{n}},
\qquad
\mathbb{P}\left(D_{n,m}^{+}\ge\delta\right)
=1-\frac{A_{n,m}}{\binom{n+m}{n}}.
\]
\end{proposition}

\begin{proof}
Under the continuous null $F=G$, every pooled ordering of the $n$ observations
from the first sample and the $m$ observations from the second sample is equally
likely. A partial ordering with $i$ observations from the first sample and $j$
from the second sample has empirical difference $i/n-j/m$. Therefore the event
$D_{n,m}^{+}<\delta$ is exactly the event that every visited lattice point
$(i,j)$ satisfies $i/n-j/m<\delta$. The recursion counts admissible paths to
$(i,j)$ by adding the admissible paths entering from $(i-1,j)$ and $(i,j-1)$,
and setting the count to zero whenever the boundary has already been hit. Since
there are $\binom{n+m}{n}$ pooled orderings, division by this total gives the
probability.
\end{proof}

The same ledger can keep both walls at once, giving an exact finite-sample
calibrator for the ordinary two-sided statistic when $n$ and $m$ are unequal.

\begin{proposition}[Two-wall recursion for unequal two-sample KS \cite{feller1968introduction}]
Let
\[
D_{n,m}=\sup_t|F_n(t)-G_m(t)|.
\]
Under $F=G$, fix $\delta>0$ and define $B_{0,0}=1$. For $0\le i\le n$ and
$0\le j\le m$, set $B_{i,j}=0$ outside the rectangle and, for $(i,j)\ne(0,0)$,
\[
B_{i,j}=
\mathbf{1}\left\{\left|\frac{i}{n}-\frac{j}{m}\right|<\delta\right\}
\left(B_{i-1,j}+B_{i,j-1}\right).
\]
Then
\[
\mathbb{P}\left(D_{n,m}<\delta\right)
=\frac{B_{n,m}}{\binom{n+m}{n}},
\qquad
\mathbb{P}\left(D_{n,m}\ge\delta\right)
=1-\frac{B_{n,m}}{\binom{n+m}{n}}.
\]
\end{proposition}

\begin{proof}
The preceding proof applies with the admissible half-plane replaced by the
strip $|i/n-j/m|<\delta$. The recursion counts paths that remain in this strip
from $(0,0)$ to $(n,m)$; division by the total number $\binom{n+m}{n}$ of
pooled orderings gives the probability.
\end{proof}

Only grid values of $|i/n-j/m|$ need be checked, so the same recursion also
returns exact critical values for any desired level.

When $n=m$, write $G_n$ for the second empirical distribution. The weighted
boundary becomes a constant level for a simple walk, and the recursion collapses
to the familiar reflection count.

\begin{theorem}[Gnedenko-Korolyuk reflection formula \cite{feller1968introduction}]
    Let $D_{n,n}^{+}$ be the one-sided Kolmogorov-Smirnov statistic:
    \[
    D_{n,n}^+=\sup_{t}\left(F_n(t)-G_n(t)\right).
    \]
    If $F=G$, then for $r=1,\ldots,n$,
    \begin{align}
        \mathbb{P}\left(D_{n,n}^+\ge\frac{r}{n}\right)
        =\frac{\binom{2n}{n-r}}{\binom{2n}{n}}.
    \end{align}
\end{theorem}
\begin{proof}
Encode each pooled ordering of the $X$ and $Y$ samples as a path with $2n$
steps: an $X$-observation contributes $+1$, and a $Y$-observation contributes
$-1$. After $2n$ steps, the path returns to $0$. The total number of such
paths is
\[
\binom{2n}{n}.
\]

The event $\{D_{n,n}^+\ge r/n\}$ is equivalent to the event that the path
hits level $r$. By the reflection principle \cite{feller1968introduction},
reflecting the path after its first hit of $r$ maps such paths bijectively
to paths that end at $2r$. Thus the number of crossing paths is
\[
\binom{2n}{n+r}=\binom{2n}{n-r}.
\]
Since all pooled orderings are equally likely under the null,
\[
\mathbb{P}\left(D_{n,n}^+\ge\frac{r}{n}\right)
=\frac{\binom{2n}{n-r}}{\binom{2n}{n}}.
\]
\end{proof}

\begin{corollary}[Feller limit \cite{feller1991introduction}]
    For fixed $r>0$, if $n\rightarrow\infty$, then
    \begin{align}
        \mathbb{P}\left(D_{n,n}^+\ge\frac{r}{\sqrt{n}}\right)\rightarrow e^{-r^2}.
    \end{align}
\end{corollary}
\begin{proof}
Let $k_n=\lceil r\sqrt n\rceil$. Since $D_{n,n}^{+}$ takes values on the grid
$\{0,1/n,\ldots,1\}$, the preceding theorem gives
\[
\mathbb{P}\left(D_{n,n}^+\ge\frac{r}{\sqrt n}\right)
=\frac{n!n!}{(n-k_n)!(n+k_n)!}.
\]
Stirling's formula, with $k_n/n\to0$ and $k_n^2/n\to r^2$, gives the limit
$e^{-r^2}$.
\end{proof}

\section{Two Sides: Compressing Crossings Into DKWM}
The two-sided one-sample KS statistic is defined as
\[
D_n=\sup_t|F_n(t)-F(t)|.
\]
The one-sided story keeps track of the first lower crossing. The two-sided
statistic asks the path to hit either wall. Exact two-sided formulas exist, but
for many applications the right object is the sharp envelope: a bound that
forgets the particular hitting location while preserving the correct exponential
scale.

\begin{theorem}[Dvoretzky-Kiefer-Wolfowitz-Massart \cite{dvoretzky1956asymptotic,massart1990tight}]
For every distribution function $F$ and every $\epsilon>0$,
\begin{align}
    \mathbb{P}\left(D_n>\epsilon\right)\le 2e^{-2n\epsilon^2}.
\end{align}
\end{theorem}

Massart \cite{massart1990tight} proved that the constant $2$ is sharp; see
also Kosorok \cite{kosorok2008introduction} for a modern empirical-process
treatment, including distribution functions with jumps. The connection with
the crossing formula above is that the one-sided event is exactly the finite
sum
\[
    \mathbb{P}(D_n^->\epsilon)
    =(1-\epsilon)^n+\epsilon\sum_{j=1}^{\lfloor n(1-\epsilon)\rfloor}
    \binom{n}{j}
    \left(\epsilon+\frac{j}{n}\right)^{j-1}
    \left(1-\epsilon-\frac{j}{n}\right)^{n-j}.
\]
Massart's argument controls this crossing sum, and its upper analogue, at the
scale $\exp(-2n\epsilon^2)$. Thus the DKWM inequality is not a separate
phenomenon in this narrative: it is the sharp uniform bound obtained after
the exact first-passage calculation is compressed into an exponential tail.

\begin{remark}[Numerical scale]
The exact calculation is small enough to be used directly. For example, when
$n=100$, the one-sided crossing formula gives
$\mathbb{P}(D_n^->0.10)=0.1266$ and
$\mathbb{P}(D_n^->0.12)=0.0517$, whereas the two-sided DKWM envelopes are
$0.2707$ and $0.1123$. For the unequal two-sample recursion, with
$(n,m)=(100,80)$ and $\delta=0.12$, the exact values are
$\mathbb{P}(D_{n,m}^{+}\ge\delta)=0.2570$ and
$\mathbb{P}(D_{n,m}\ge\delta)=0.5056$. These values are reproduced by the
exact-crossing script included in the source package.
\end{remark}

\section{Two-Sample Limits}
For the two-sided two-sample case, define
\[
D_{n,m}=\sup_t|F_n(t)-G_m(t)|.
\]
Under the null hypothesis $F=G$, its limiting distribution is
\begin{align}
    \lim_{n,m\rightarrow\infty}\mathbb{P}\left( \sqrt{\frac{nm}{n+m}}D_{n,m}> \lambda \right)&=2\sum_{k=1}^\infty (-1)^{k-1}e^{-2k^2\lambda^2}.
\end{align}
The two-wall recursion in Section~3 is exact but finite. The display above is
what remains when the lattice is scaled and allowed to converge: the path
becomes a Brownian bridge and the alternating exponential series replaces path
counting.
This limit can be shown via Donsker's theorem or empirical process theory
\cite{kosorok2008introduction}. For richer suprema over function classes, the
limiting distribution may also depend on the common distribution $F=G$.

\section{Composite Nulls: Where Exact Counting Stops}
The crossing ledger has a trap door. Everything above relies on knowing the
null distribution before seeing the data. Once parameters are estimated from
the same data, the transformed observations no longer behave like independent
uniforms under a fixed distribution. The path being tested has been altered by
the fitting step, so exact distribution-free counting stops.

In some scenarios, the i.i.d. sample $X_1,\ldots,X_n$ may come from $F_\theta$, where $\theta$ is a vector of unknown parameters. For example, $F_\theta$ may be the normal distribution and $\theta=(\mu,\sigma^2)^T$ contains the mean and variance. We first estimate $\theta$, for instance by maximum likelihood estimation or, more generally, by M-estimation. Define
\begin{align}
    Z_n(t,\theta)&=\sqrt{n}\left(F_n(t)-F_\theta(t)\right)\\
    Z_n(t,\widehat{\theta})&=\sqrt{n}\left(F_n(t)-F_{\widehat{\theta}}(t)\right).
\end{align}
The problem is to find the limiting distribution of $Z_n(t,\widehat{\theta})$. Setting aside measurability technicalities, which for general M-estimation are often handled through outer measurability, the following heuristic argument is due to Babu and Rao \cite{babu2004goodness}; see also \cite{durbin1973weak,van2000asymptotic, del2007lectures}.
\begin{align*}
    Z_n(t,\widehat{\theta})&=Z_n(t,\theta)+\sqrt{n}\left(F_\theta(t)-F_{\widehat{\theta}}(t)\right)\\
    &=Z_n(t,\theta)-\frac{\partial F_\theta(t)}{\partial\theta}^T\sqrt{n}(\widehat{\theta}-\theta)+o_{\mathbb{P}_\theta}(1)\\
    &=Z_n(t,\theta)-\frac{\partial F_\theta(t)}{\partial\theta}^T\frac{1}{\sqrt{n}}\sum_{i=1}^n\psi_\theta(X_i)+o_{\mathbb{P}_\theta}(1)
\end{align*}
where we have assumed that 
\[
\sqrt{n}(\widehat{\theta}-\theta)=\frac{1}{\sqrt{n}}\sum_{i=1}^n\psi_\theta(X_i)+o_{\mathbb{P}_\theta}(1),
\]
where $\psi_\theta$ is the influence function and $\mathbb{E}\psi_\theta(X_i)=0$. The weak limit can then be derived from the joint limit of
\[
\left(Z_n(\cdot,\theta),\frac{1}{\sqrt{n}}\sum_{i=1}^n\psi_\theta(X_i)\right).
\]
Donsker's theorem gives $Z_n(\cdot,\theta)\Rightarrow B(F_\theta(\cdot))=B\circ F_\theta$, where $B$ is a standard Brownian bridge, while the central limit theorem gives $\frac{1}{\sqrt{n}}\sum_{i=1}^n\psi_\theta(X_i)\Rightarrow \mathcal{N}(0,\Sigma(\theta))$, where $\Sigma(\theta)=\mathbb{E}\psi_\theta(X_1)^{\otimes 2}$. The covariance can be computed from
\[
\Cov(\mathbb{I}(X_1\le t),\psi_\theta(X_1))=\int_{-\infty}^t\psi_\theta(x)F_\theta(dx):=h(t,\theta).
\]
Hence, we have finite-dimensional weak convergence. Tightness then yields weak convergence of the stochastic process $Z_n(t,\widehat{\theta})$ \cite{babu2004goodness}; the limit is a mean-zero Gaussian process $Z(\cdot,\theta)$ with covariance
\begin{align}\label{eq:cov}
    \Cov\left(Z(s,\theta),Z(t,\theta)\right)
    &=F_\theta(s\wedge t)-F_\theta(s)F_\theta(t)\\
    &\quad-h(t,\theta)^T\frac{\partial F_\theta(s)}{\partial \theta}\nonumber\\
    &\quad-h(s,\theta)^T\frac{\partial F_\theta(t)}{\partial \theta}\nonumber\\
    &\quad+\frac{\partial F_\theta(s)}{\partial\theta}^T
      \Sigma(\theta)
      \frac{\partial F_\theta(t)}{\partial\theta}.\nonumber
\end{align}

To sum up, we have the following theorem.
\begin{theorem}[Durbin-Babu-Rao \cite{durbin1973weak,babu2004goodness}]
    Let $X_1,\ldots,X_n$ be a random sample with common distribution function $F_\theta$ and $\theta\in\mathbb{R}^k$. Let $\widehat{\theta}$ be an asymptotically linear estimator of $\theta$ with influence function $\psi_\theta$. Define
    \[
    D_n(\widehat{\theta})=\sup_{t}|F_n(t)-F_{\widehat{\theta}}(t)|.
    \]
    Then, at every continuity point $\epsilon$ of the limiting law,
    \[
    \mathbb{P}\left(\sqrt{n}D_n(\widehat{\theta})>\epsilon\right)
    \longrightarrow
    \mathbb{P}\left(\sup_t|Z(t,\theta)|>\epsilon\right),
    \]
    where $Z(\cdot,\theta)$ is a mean-zero Gaussian process with covariance specified in \eqref{eq:cov}. This probability can be approximated by bootstrap or Monte Carlo simulation.
\end{theorem}

\begin{remark}[Bad examples for the naive KS calibration]
The theorem explains a common trap. If one estimates the parameter and then
uses the classical Brownian-bridge critical value as if $\theta$ were known, the
test calibrates the wrong process. Table~\ref{tab:composite-null} gives a small
Monte Carlo warning at nominal level $5\%$ with $n=100$. The classical cutoff is
$1.358/\sqrt{n}$; the bootstrap column uses a parametric bootstrap in which
parameters are re-estimated in every bootstrap sample. The simulation uses
seed 20260520, 2000 Monte Carlo samples, and 399 bootstrap samples per run;
the reproducible script is included as \texttt{scripts/ks\_composite\_simulation.py}.
\end{remark}

\begin{table}[t]
\centering
\caption{Empirical rejection probabilities under the null, based on 2000 Monte
Carlo samples. The bootstrap column uses 399 bootstrap resamples per Monte Carlo
sample.}
\label{tab:composite-null}
\begin{tabular}{lcc}
\hline
Null model & KS cutoff & Refit bootstrap\\
\hline
Known $N(0,1)$ & 0.051 & --\\
Fitted normal $(\mu,\sigma^2)$ & 0.000 & 0.044\\
Fitted exponential rate & 0.007 & 0.052\\
\hline
\end{tabular}
\end{table}

\begin{remark}[Real-data diagnostics]
Two small data checks make the calibration point concrete, using the
\texttt{trees} and \texttt{attenu} data sets bundled with R \cite{r2026}. For
tree volume fitted by a lognormal model, $D_n(\widehat\theta)=0.101$, the
simple-null $p$-value is $0.913$, and the refitted-bootstrap $p$-value is
$0.590$. For earthquake acceleration, the corresponding values are $0.115$,
$0.016$, and $<0.001$. The important feature is not the lognormal model itself,
but the calibration step: after fitting, the simple-null and refitted-bootstrap
$p$-values answer different questions. The reproducible script is included in
the source package.
\end{remark}

\begin{remark}[Cases outside the linear expansion]
There are also models where the Durbin-Babu-Rao expansion itself should not be
used without extra work. A simple example is a distribution with parameter-dependent
support, such as $\operatorname{Unif}(0,\theta)$ with $\theta$
estimated by the sample maximum: the estimator has an $n$-rate, non-Gaussian
limit and is not asymptotically linear in the form assumed above. Mixture
models near nonidentifiable parameter values create a different version of the
same problem. In such cases, a refitted simulation scheme may still be useful,
but the limiting Gaussian projection in \eqref{eq:cov} is no longer the right
justification.
\end{remark}

\begin{remark}
In the case of maximum likelihood estimation, the influence function is
\[
\psi_\theta(x)=I^{-1}_\theta s_\theta(x),
\]
where $I^{-1}_\theta$ is the inverse of Fisher information and $s_\theta(x)$ is the score function evaluated at $x$.
\end{remark}

\begin{example}
Let $X\sim N(\mu,\sigma^2)$. The influence function for the maximum likelihood estimator is
\[
\psi_\theta(x)=
\begin{pmatrix}
    x-\mu\\
    (x-\mu)^2-\sigma^2
\end{pmatrix},
\]
and
\[
\Sigma(\theta)=I_\theta^{-1}=
\begin{pmatrix}
    \sigma^2 & 0\\
    0 & 2\sigma^{4}
\end{pmatrix}.
\]
\end{example}

\section{Discussion}
The starting question was deliberately local: where does the empirical process
first cross a boundary? The crossing ledger developed here gives a compact
answer. In the one-sample one-sided case, the probability integral transform
exposes a finite lattice and yields an exact, stable tail calculation. In the
two-sample case, pooled orderings become paths; one-wall and two-wall recursions
give exact finite-sample probabilities for arbitrary $n$ and $m$, while
reflection explains the special balanced one-sided closed form. For two-sided
one-sample statistics, the same picture no longer gives a single first-passage
distribution, but it does explain why the DKWM inequality is the sharp
exponential envelope for crossing either boundary.

The ledger is also useful because it records where exactness fails. Under a
composite null, parameter estimation changes the path being tested; the
resulting process is a projected Gaussian limit rather than a distribution-free
Brownian bridge. The small simulation in Table~\ref{tab:composite-null} is meant
as a warning rather than a benchmark: a naive KS cutoff can have essentially
zero rejection under fitted models, while refitted bootstrap calibration returns
the size to the intended scale. The two data checks show the corresponding
diagnostic use: refitted calibration can distinguish a benign fitted example
from a clear lack of fit.

Thus the practical message is simple. When the null is simple and the statistic
is one-sided, compute the exact crossing probability. For two-sample statistics,
use the one-wall or two-wall recursion when finite-sample calibration matters.
When a two-sided one-sample summary is enough, DKWM is the sharp exponential
compression. When parameters are estimated, refit in the simulation or work with
the projected empirical-process limit.

\section*{Declarations}

\paragraph{Declaration of generative AI and AI-assisted technologies in the writing process}
During the preparation of this work the authors used OpenAI's ChatGPT/Codex
to assist with language editing, LaTeX formatting, and consistency checks.
After using this tool, the authors reviewed and edited the content as needed
and take full responsibility for the content of the publication.

\paragraph{Competing interests}
The authors declare no competing interests related to this methodological note.

\paragraph{Funding}
No specific funding was reported for this work.

\paragraph{Data and code availability}
The real-data diagnostics use the \texttt{trees} and \texttt{attenu} data sets
distributed with R. The numerical results in Table~\ref{tab:composite-null}
are based on synthetic Monte Carlo samples. The source package includes scripts
for the exact crossing calculation, the composite-null simulation, and the
real-data diagnostics.

%\section{Acknowledgments}

\bibliographystyle{elsarticle-num}

\bibliography{references, pkg, historical}

\end{document}